\documentclass[12pt]{amsart}
\usepackage[margin=1in]{geometry}                % See geometry.pdf to learn the layout options. There are lots.
\usepackage{graphicx}
\usepackage{amssymb}
\usepackage{amsthm}
\usepackage{mathrsfs}
\usepackage{caption}
\usepackage{xcolor}
\usepackage{bbm}

\newtheorem{theorem}{Theorem}[section]
\newtheorem{proposition}[theorem]{Proposition}
\newtheorem{lemma}[theorem]{Lemma}
\newtheorem{definition}[theorem]{Definition}

\numberwithin{equation}{section}

\newcommand{\R}{\mathbb{R}}
\newcommand{\N}{\mathbb{N}}
\newcommand{\Z}{\mathbb{Z}}

\newcommand{\F}{\mathbb{F}}
\renewcommand{\i}{\vec{\imath}}
\renewcommand{\j}{\vec{\jmath}}

\begin{document}

\title{BRK-type sets over finite fields} 
\author{Charlotte Trainor}	
	
\date{\today}

\maketitle

\begin{abstract}

 A Besicovitch-Rado-Kinney (BRK) set in $\R^n$ is a Borel set that  contains a $(n-1)$-dimensional sphere of radius $r$, for each $r>0$. It is known that such sets have Hausdorff dimension $n$ from the work of Kolasa and Wolff. In this paper, we consider an analogous problem over a finite field, $\F_q$. We define BRK-type sets in $\F_q^n$, and establish lower bounds on the size of such sets using techniques introduced by Dvir's proof of the finite field Kakeya conjecture. 

\end{abstract}

\section{Introduction}

Let $n\geq2$. We call $E\subset\R^n$ a \emph{Kakeya set} if for every direction $v\in\mathbb{S}^{n-1}$, the set $E$ contains a unit line segment in direction $v$, i.e., there is some point $a\in\R^n$ so that
\[ \{a+tv:t\in[0,1]\}\subset E. \]
The \textit{Kakeya conjecture} posits that a Kakeya set in $\R^n$ has Hausdorff dimension $n$. While the conjecture has been established by Davies for dimension $n=2$ \cite{Davies}, the conjecture in dimensions $n\geq3$  remains a major open problem in harmonic analysis.

In his influential survey on the Kakeya problem, Wolff formulated the \textit{finite field Kakeya conjecture} as a simpler prototype for the Euclidean problem \cite{Wolffsurvey}. Let $\F_q$ be a finite field with $q$ elements. Then a set $S\subset\F_q^n$
is a \emph{Kakeya set} if for every non-zero $b\in\F_q^n$, there is some $a\in\F_q^n$ so that 
\[ \{a+tb:t\in\F_q\}\subset S. \]
Wolff conjectured that there is some constant $C_n>0$ so that for any Kakeya set $S\subset\F_q^n$, $|S|\geq C_nq^n$, where $|S|$ denotes the number of elements in $S$. 

In 2008, Dvir published a proof of the finite field Kakeya conjecture using, and popularizing, the \textit{polynomial method} \cite{Dvir}.
An improved bound was obtained by Dvir, Kopparty, Saraf and Sudan  using a variant of the polynomial method, known as the \textit{method of multiplicities}~\cite{methodmult}. Ellenberg, Oberlin, and Tao \cite{EOT} and Hickman and Wright \cite{HW} formulated the Kakeya problem over more general rings, such as the $p$-adic integers and the ring of integers mod $N$, $\Z_N$. Using a variant of the polynomial method, Dhar and Dvir established the  Kakeya conjecture over $\Z_N$ for square free $N$ \cite{DD}, and Arsovski established the $p$-adic Kakeya conjecture \cite{padic}. In \cite{Dhar}, Dhar combined the methods from~\cite{DD} and~\cite{padic} to solve the Kakeya problem over $\Z_N$, for general $N$.

\subsection{Besicovitch-Rado-Kinney sets} 

We say that $E\subset\R^n$ is a \emph{Besicovitch-Rado-Kinney (BRK) set} if $E$ is Borel, and for each $r>0$, there is some point $a\in\R^n$ so that
\[ \{a+rx:|x|=1\}\subset E. \]
So, BRK sets contain a sphere of radius $r$, for every such $r$. Both Besicovitch-Rado \cite{BR} and Kinney \cite{Kinney} showed that there exists BRK sets of measure zero.

As a variant of the Kakeya problem, we may consider the dimension of BRK sets in $\R^n$. In \cite{KolasaWolff}, Kolasa and Wolff quantified the Hausdorff dimension of a BRK set via estimates on a maximal operator known as Wolff's circular maximal function. Using this method, they showed that for $n\geq3$, BRK sets in $\R^n$ have Hausdorff dimension $n$.
Later, in \cite{WolffBRK}, Wolff showed that all BRK sets in $\R^2$ have Hausdorff dimension~$2$.

%We could also generalize BRK sets to higher dimension, and consider compact subsets of $\R^n$ that, for each radius $r$ in some fixed interval, contain a sphere of radius $r$. The size of these sets in dimension $n\geq 3$ is more straightforward than in dimension $n=2$, and Kolasa-Wolff proved that for $n\geq 3$, such sets have full Hausdorff dimension. 

Kolasa and Wolff also introduced analogous problems in dimension $n=2$ with circles replaced by more general curves. In higher dimensions, BRK sets have been generalized in~\cite{Sawyer, Wisewell}, who construct measure zero subsets of $\R^n$ given by unions of $d$-dimensional surfaces.

%For functions $\Phi$ satisfying certain smoothness conditions, we can consider sets $E\subset \R^2$ so that for each $r\in (0,1)$, there is some $x$ so that 
%\[ \{y\in\R^2: \Phi(x,y)=r\}\subset E. \]
%For some classes of functions $\Phi$, such sets must have full Hausdorff dimension \cite{Zahl}.

\subsection{Analogous problem over finite fields} Let $\F_q$ be a finite field with $q$ elements. We define BRK-type sets over $\F_q$ and determine a lower bound on their size, in analogy with the question of the dimension of BRK-type sets over $\R$.
We define BRK-type sets over $\F_q$ as follows:

\begin{definition}
\label{def-brk}
 Let $g\in\F_q[s_1,\dots,s_{n-1}]$ be a homogeneous polynomial of degree $\ell\geq2$, and let $\mathcal{P}_g\subset\F_q[s_1,\dots,s_{n-1}]$ be the set of all polynomials with homogeneous part of highest degree equal to $g$. We say that $S\subset \F_q^n$ is a \emph{BRK-type set of degree $\ell$} if for any $\rho\in\F_q$, there exist $a=a(\rho)\in\F_q^n$ and $g_\rho\in\mathcal{P}_g$ so that 
\begin{equation}
\label{brkdef}
\{a+\rho (\lambda, g_\rho(\lambda)):\lambda\in\F_q^{n-1}\}\subset S.
\end{equation}
%Then we say $S$ is a \emph{BRK-type set of degree $\ell$.}
\end{definition}
%% REPHRASE THIS A LITTLE. Define a family? g+\script{P}.
We require $\ell\geq2$, as if we allowed $\ell=1$, then a hyperplane in $\F_q^n$ would be a BRK-type set of degree $1$. 

For example, suppose $S\subset\F_q^2$ is a BRK-type set satisfying~(\ref{brkdef}) with $g_\rho(s)=s^2$ for all $\rho$. Then for any non-zero $\rho\in\F_q$, there is some $a=(a_1,a_2)$ so that $S$ contains all solutions to the equation
\[ y-a_2=\rho^{-1}(x-a_1)^2. \]
Thus for any non-zero $c\in\F_q$, $S$ contains a parabola with abscissa equal to $c$.  We can also define BRK sets where the polynomials $g_\rho$ depend on $\rho$; say for $n=2$, we could take $g_\rho(s)=s^3+h_\rho(s)$ where $h_\rho$ is a polynomial of degree at most $2$.

%% cal1 EXAMPLE 1.
%% MAKE EXAMPLE 2 - cubic with different g_rho.

Previous generalizations of Kakeya and BRK sets have been studied in~\cite{EOT},~\cite{sphere} and~\cite{cone}. In~\cite{sphere}, Makhul, Warren and Winterhof define a sphere of radius $r$ in $\F_q^n$ to be a subset 
\[ \{x\in\F_q^n: (x_1-a_1)^2+\dots+(x_n-a_n)^2=r\}. \]
Notice that this set cannot be written in the form~(\ref{brkdef}), which we consider. They prove lower bounds on the size of sets containing spheres of any radius using known bounds on the number of solutions to quadratic equations over $\F_q$. 

In~\cite{cone}, Warren and Winterhof define parabolas, hyperbolas, and ellipses in $\F_q^n$, and consider subsets of $\F_q^n$ containing many such \textit{curves}. This differs from our definition, as we can think of~(\ref{brkdef}) as a \textit{hypersurface} in $\F_q^n$. 

Our main result is the following:

\begin{theorem}
\label{thm}
Let $\ell\in \N$ with $2\leq \ell<q$. Let $S\subset\F_q^n$ be a BRK-type set of degree $\ell$, as defined in Definition~\ref{def-brk}.
Then 
\[ |S|\geq \dfrac{(q-1)^n}{(\ell+1-2\ell/q)^n }.\]
\end{theorem}

We will prove this theorem by reducing it to a ``homogeneous version'' of the problem. In this statement, we will use some terms related to polynomials that we define in detail in section 2; in particular, Hasse derivatives are defined in Definition~\ref{def-Hasse}. 
%% WRITE DEFN OF ZERO POLY BEFORE THIS. %%
\begin{proposition}
\label{prop}
Let $n\geq 2$. Let $Q\in \F_q[x_1,\dots,x_n]$ be a polynomial of degree $d$. Suppose that there are  $\ell,m\in\Z_{\geq0}$ with $\ell\geq2$ so that each exponent $\alpha\in \Z_{\geq0}^n$ appearing in $Q$ satisfies
\[ \alpha_1+\dots+\alpha_{n-1}+\ell\alpha_n=m. \]
Let $f\in\F_q[s_1,\dots,s_{n-1}]$ be a homogeneous polynomial of degree $\ell$.

For $\beta\in\Z_{\geq0}^n$, let $Q_\beta=Q^{(\beta)}$, the $\beta$-th Hasse derivative of $Q$.  Let $k\in\N$ be such that $d<k(q-1)$ and the following holds: for all non-zero $\rho\in\F_q$ and all $\beta\in\Z_{\geq0}^n$ with $\beta_1+\dots+\beta_n<k$, the polynomial
\[ Q_{\beta,\rho}(s):=Q_\beta(\rho(s,f(s)))\]
is the zero polynomial (meaning all its coefficients are zero). 
Then $Q$ is the zero polynomial.
\end{proposition}

Before proving these statements, we give some preliminary results in sections 2 and 3, and as a warm up, consider a special case in section~\ref{secdim2}. 
In section~\ref{secthm}, we prove Theorem~\ref{thm} assuming Proposition~\ref{prop}, and finally in section~\ref{secprop}, we prove Proposition~\ref{prop}.

\section{Preliminaries}
Consider $\F_q[x_1,\dots,x_n]$, the space of $n$-variate polynomials with coefficients in $\F_q$. 
We will let $x=(x_1,\dots,x_n)$ throughout the remainder of the paper. 
For $\alpha=(\alpha_1,\dots,\alpha_n)\in\Z_{\geq0}^n$, we define
\[ x^\alpha = x_1^{\alpha_1}\cdots x_n^{\alpha_n}. \]
We also define $|\alpha|:=\alpha_1+\dots+\alpha_n$. Then $x^\alpha$ is a monomial of degree $|\alpha|$. 

We say that a polynomial $P(x)=\sum_\alpha c_\alpha x^\alpha$ is \textit{non-zero }if there is some $c_\alpha$ that is non-zero, and otherwise, call $P$ the \emph{zero polynomial}. If $P$ is non-zero,  the degree of $P$, denoted $\text{deg}(P)$, is $\max\{|\alpha|:c_\alpha\neq0\}$.
For convenience, we define the degree of the zero polynomial to be $-\infty$.  Notice that it's possible for a non-zero polynomial to vanish on all of $\F_q^n$; one such example is $P(x)=x_1^q-x_1$.

Dvir's original proof of the finite field Kakeya conjecture starts off by assuming that a Kakeya set $S$ is small enough so that there must be a non-zero polynomial of degree at most $q-1$ vanishing on it. Then, using the structure of the Kakeya set, he derives a contradiction, concluding that $|S|\geq c_n q^n$ for an explicit constant $c_n$. 

In~\cite{methodmult}, Dvir, Kopparty, Saraf, and Sudan improve this bound (in particular, replace the constant $c_n$ with $c^n$, where $c$ is independent of $n$) using the \textit{method of multiplicities.} In their argument, given a Kakeya set $S$, they consider polynomials vanishing on $S$ with high multiplicity. We describe the tools needed for this method in the following subsection. 

\subsection{Hasse Derivatives and the Method of Multiplicities}

\begin{definition}
\label{def-Hasse}
Let $x=(x_1,\dots,x_n)$, $y=(y_1,\dots,y_n)$, and $P\in \F_q[x]$. Let $\beta\in \Z_{\geq0}^n$. The \emph{$\beta$-th Hasse derivative of $P$}, which we denote by $P^{(\beta)}$, is the coefficient of $y^\beta$ in $P(x+y)$. Thus we may write
\[ P(x+y)=\sum_\beta P^{(\beta)}(x)y^\beta. \]
Given a point $a\in \F_q^n$, the \emph{multiplicity of $P$ at $a$} is 
\[ \text{mult}(P,a)=\max\{M\in\Z_{\geq0}:P^{(\beta)}(a)=0 \text{ for all } \beta\in\Z_{\geq0}^n \text{ with } |\beta|<M\}, \]
with the convention that if the maximum of this set does not exist, then $\text{mult}(P,a)=\infty$.  Moreover, for $A\subset\F_q^n$, we say that the polynomial $P$ \emph{vanishes on $A$ with multiplicity $M$} if $\text{mult}(P,a)\geq M$ for all $a\in A$. 
\end{definition}

In~\cite{methodmult}, Dvir et al prove the following properties of and results relating to multiplicities of polynomials. We will also use these results in our arguments. 

\begin{lemma}[\cite{methodmult}, Lemma 5]
\label{multHD}
Let $P\in \F_q[x_1,\dots,x_n]$ and $\beta\in\Z_{\geq0}^n$. Then for $a\in \F_q^n$, 
\[ \text{mult}(P^{(\beta)},a)\geq \text{mult}(P,a)-|\beta|. \]
\end{lemma}

\begin{lemma}[\cite{methodmult}, Proposition 6]
\label{multcomp}
Let $P\in \F_q[x_1,\dots,x_n]$ and $h\in (\F_q[t])^n$ (where $t$ is a single variable). Then for $\lambda\in\F_q$,
\[ \text{mult}(P\circ h, \lambda) \geq \text{mult}(P, h(\lambda)). \]
\end{lemma}

\begin{lemma}[\cite{methodmult}, Proposition 4]
\label{lem-sum}
Let $P,Q\in\F_q[x_1,\dots,x_n]$ and $\beta\in\Z_{\geq0}^n$. Then $(P+Q)^{(\beta)}=P^{(\beta)}+Q^{(\beta)}$. 
\end{lemma}

The following lemma describes, for a particular degree $D$ and multiplicity $M$, how small a set $A$ must be to guarantee that there exists a polynomial of degree at most $D$ vanishing on $A$ with multiplicity $M$.

\begin{lemma}[\cite{methodmult}, Lemma 8]
\label{exists}
Let \(A\subset \F_q^n\). If 
\[ {M+n-1\choose n}\cdot|A|<{D+n\choose n} \]
then there exists a non-zero polynomial \(P\in \F_q[x_1,\dots,x_n]\) of degree at most $D$ vanishing on~\(A\) with multiplicity $M$. 

\end{lemma}

\subsection{Schwartz-Zippel Lemma and a key application} The next lemma is a strengthened version of the Schwartz-Zippel Lemma.

\begin{lemma}[\cite{methodmult}, Proposition 10]
\label{SZ}
Let \(A\subset \F_q\). Let \(P\in \F_q[x_1,\dots,x_n]\) be a non-zero polynomial of degree \(d\). Then
\[ \sum_{a\in A^n}\emph{mult}(P,a)\leq d|A|^{n-1}. \]
Consequently, if $P$ is a polynomial vanishing on $\F_q^n$ with multiplicity $M$, and $\text{deg}(P)<Mq$, then $P$ is the zero polynomial.
\end{lemma}

We will also need the following lemma, which is closely related to a statement established in the proof of Theorem~11 in~\cite{methodmult}.

\begin{lemma}
\label{derivszero}
Let $\ell\in \N$ with $2\leq \ell<q$, and let $g\in\F_q[x_1,\dots,x_{n-1}]$ be a polynomial of degree $\ell$. For $\rho\in\F_q$ and $a\in \F_q^n$, let 
\[C=\{a+\rho(\lambda,g(\lambda)):\lambda\in\F_q^{n-1}\}. \]
Let $k,D,M\in\N$ be such that
\begin{equation}
\label{ineqMDgen}
\ell(D-w)<(M-w)q  \quad \text{for all} \quad 0 \leq w< k.
\end{equation} Suppose that $P$ is a non-zero polynomial of degree at most $D$ vanishing on $C$ with multiplicity $M$. Let $P^{(\beta)}$ denote the Hasse derivative of $P$ of order $\beta\in\Z_{\geq0}^n$, and let 
\[ P_{\beta,\rho}(s)=P^{(\beta)}(a+\rho(s,g(s))).\]
Then for $|\beta|<k$, $P_{\beta,\rho}$ is the zero polynomial.
\end{lemma}

\begin{proof}
%Before proceeding, observe that for our choice of $M$ and $D$, we have
%\begin{equation}
%\label{ineqMDgen}
%\ell(D-w)<(M-w)q  \quad \text{for all} \quad 0 \leq w\leq k.
%\end{equation}
%This follows after verifying the inequalities $\ell D<Mq$ and $\ell(D-k)<(M-k)q$, the first of which relies on the assumption that $\ell<q$.

Choose some $\beta\in \Z_{\geq0}^n$ with $|\beta|<k$. Let $P_\beta=P^{(\beta)}$. Notice that $P_\beta$ has degree at most $D-|\beta|$. By Lemma~\ref{multHD}, $P_\beta$ vanishes on $C$ with multiplicity $M-|\beta|$. Applying this with Lemma~\ref{multcomp}, we see that 
\[ P_{\beta,\rho}(s) := P_\beta(a+\rho(s,g(s))) \]
vanishes on $\F_q$ with multiplicity $M-|\beta|$. Moreover, $P_{\beta,\rho}$ has degree at most $\ell(D-|\beta|)$. 
Then 
\[ \sum_{\lambda\in\F_q^{n-1}}\text{mult}(P_{\beta,\rho},\lambda)\geq q^{n-1}(M-|\beta|)>q^{n-2}\ell(D-|\beta|)\geq q^{n-2}\text{deg}(P_{\beta,\rho})\]
where the second inequality is given by~(\ref{ineqMDgen}). Thus, by Lemma~\ref{SZ}, $P_{\beta,\rho}$ is the zero polynomial.     
\end{proof}

\subsection{Binomial coefficients and explicit form for Hasse derivatives}

For our arguments, we will need to use the explicit form for a Hasse derivative. First, we define binomial coefficients of multi-indices. For $\alpha,\beta\in \Z_{\geq0}^n$, let
\begin{equation}
\label{bincoeff}
{\alpha \choose \beta}=\prod_{i=1}^n{\alpha_i\choose \beta_i}.
\end{equation}
We use the convention that ${a\choose b}=0$ for $b>a$.  Then 
\begin{equation} 
\label{conv}
{\alpha\choose\beta}=0 \quad \text{if there is some $i$ so that } \alpha_i<\beta_i. 
\end{equation}
We will need the following result for Binomial coefficients, known as Vandermonde's Identity:
\begin{lemma}
\label{lemma:bincoeff}
Let $\alpha\in\Z_{\geq0}^n$ and $w\in\Z_{\geq0}$. Then 
\[ {|\alpha|\choose w}= \sum_{\beta\in\Z_{\geq0}^n,|\beta|=w}{\alpha\choose \beta}, \]
with the convention in~(\ref{conv}).
\end{lemma}
%\begin{proof}
%Consider, for a variable $t$, the equation 
%\[ (1+t)^{|\alpha|}=\prod_{i=1}^n(1+t)^{\alpha_i}.\]
%We can establish the lemma by applying the Binomial theorem to $(1+t)^{|\alpha|}$ and $(1+t)^{\alpha_i}$ for each $i$, and comparing the coefficients of $t^w$ after simplification. 
%\end{proof}

Now we give an explicit form for Hasse derivatives.

\begin{lemma}
\label{expHasse}
Let $P\in\F_q[x_1,\dots,x_n]$ and $x=(x_1,\dots,x_n)$. Suppose that $P$ takes the form 
\[ P(x)=\sum_\alpha c_\alpha x^\alpha. \]
Then for $\beta\in\Z_{\geq0}^n$, we have 
\[ P^{(\beta)}(x)=\sum_\alpha c_\alpha{\alpha\choose\beta}x^{\alpha-\beta} \]
where $\alpha-\beta=(\alpha_1-\beta_1,\dots,\alpha_n-\beta_n)$, and following the convention in~(\ref{conv}).
\end{lemma}

\begin{proof}
Let $y=(y_1,\dots,y_n)$. Applying the Binomial theorem and the convention that ${a\choose b}=0$ for $b>a$, we have
\[ P(x+y)=\sum_\alpha c_\alpha (x+y)^\alpha =\sum_\alpha c_\alpha \prod_{i=1}^n \sum_{\gamma_i\in\Z_{\geq0}}{\alpha_i \choose \gamma_i}x_i^{\alpha_i-\gamma_i}y_i^{\gamma_i}. \]
Letting $\gamma=(\gamma_1,\dots,\gamma_n)$, we have
\[ P(x+y)=\sum_\alpha c_\alpha\sum_{\gamma}{\alpha\choose \gamma}x^{\alpha-\gamma}y^\gamma. \]
The lemma follows by picking out the coefficient of $y^\beta$, which is by definition $P^{(\beta)}$. 
\end{proof}

\subsection{The lexicographical order}

We will use the lexicographical ordering $\preceq$ on exponents $\alpha,\beta\in\Z_{\geq0}^n$: we say $\alpha\prec\beta$ if $\alpha_j<\beta_j$, and $\alpha_i=\beta_i$ for all $1\leq i<j$, and write $\alpha\preceq\beta$ if $\alpha\prec\beta$ or $\alpha=\beta$. The following statement is well known, although we prove it here for completion.

\begin{proposition}
\label{lexwellord}
The lexicographical ordering $\preceq$ is a well ordering on $\Z_{\geq0}^n$. 
\end{proposition}
\begin{proof}
We induct on $n$. The base case $n=1$ is immediate. Now suppose the statement is true for $n\geq1$, and let $S\subset\Z_{\geq0}^{n+1}$ be non-empty. Let 
\[ S_1=\{(a_1,\dots,a_n): (a_1,\dots,a_n,a_{n+1})\in S \text{ for some }a_{n+1}\} \]
and
\[ S_2=\{a_{n+1}: (a_1,\dots,a_n,a_{n+1})\in S \text{ for some }(a_1,\dots,a_n)\}. \]
Then $S_1$ and $S_2$ are non-empty, and by the inductive hypothesis and the 
base case respectively, $S_1$ has a least element $(s_1,\dots,s_n)$, and $S_2$ has a least element $s_{n+1}$. Using the definition of $\preceq$, we may show that $(s_1,\dots,s_n,s_{n+1})$ is the least element of $S$. 

\end{proof}

\begin{lemma}
\label{lem-powerf}
Let $f\in\F_q[x_1,\dots,x_n]$ and suppose that $\alpha\in\Z_{\geq0}^n$ is such that $f=bx^{\alpha}+g$, where $b\neq0$, and every exponent in $g$ is strictly greater than $\alpha$ with respect to $\preceq$. Let $k\in\Z_{\geq0}$. Then there is some polynomial $g_k$ so that 
\[ f^k=b^kx^{k\alpha}+g_k\]
and each exponent in $g_k$ is strictly greater than $k\alpha=(k\alpha_1,\dots,k\alpha_n)$ with respect to $\preceq$. 
\end{lemma}

\begin{proof}
If $k=0$, then $f^k=1=x^0$, so its smallest exponent is $k\alpha=\vec{0}$. 
Thus we assume $k>0$. Let $\{\alpha^{(0)},\alpha^{(1)},\dots,\alpha^{(m)}\}$ be the exponents appearing in $f$, with $\alpha^{(j)}=(\alpha^{(j)}_1,\dots,\alpha^{(j)}_n)$. Without loss of generality, assume $\alpha=\alpha^{(0)}$. 
From the multinomial theorem, the exponents appearing in $f^k$ are 
\begin{equation} 
\label{expinpow}
e(\gamma):=(k-|\gamma|)\alpha+\gamma_1\alpha^{(1)}+\dots+\gamma_m\alpha^{(m)}
\end{equation}
for $\gamma\in\Z_{\geq0}^m$, $|\gamma|\leq k$. 
We will show that if $\gamma$ is not the zero vector, then $e(\gamma)\succ k\alpha$.

Let $\gamma\neq\vec{0}$. For $j=1,\dots,m$, let $i_j$ be the smallest index $i\in\{1,\dots,n\}$ where $\alpha^{(j)}_i$  differs from $\alpha_i$; then $\alpha^{(j)}_i=\alpha_i$ for each $1\leq i < i_j$, and $\alpha^{(j)}_{i_j}>\alpha_{i_j}$. Let $S_\gamma=\{j:\gamma_j\neq0\}$, and take
\[ \ell = \min\{i_j: j\in S_\gamma\}. \]
Let $e_i(\gamma)$ be the $i^{th}$ component of the exponent $e(\gamma)$. 
For $1\leq i <\ell$, we have by choice of $\ell$ that
\[ e_i(\gamma)=(k-|\gamma|)\alpha_i+\sum_{j\in S_\gamma} \gamma_j\alpha^{(j)}_i =(k-|\gamma|)\alpha_i+\sum_{j\in S_\gamma} \gamma_j\alpha_i=k\alpha_i. \]
Now consider $i=\ell$. Suppose $u\in S_\gamma$ is such that $i_u=\ell$, so that $\alpha^{(u)}_\ell > \alpha_\ell$. Then
\[ e_\ell(\gamma)=(k-|\gamma|)\alpha_\ell+ \sum_{j\in S_\gamma} \gamma_j\alpha^{(j)}_\ell \geq(k-|\gamma|)\alpha_\ell+\gamma_u(\alpha_\ell+1)+\sum_{j\in S_\gamma\setminus\{u\}} \gamma_j\alpha_\ell>k\alpha_\ell \]
where the last inequality holds as $\gamma_u\neq0$. 
Therefore $e(\gamma)\succ k\alpha$, as desired. 
\end{proof}

\begin{lemma}
\label{lem-xbf}
Let $f\in\F_q[x_1,\dots,x_n]$ and suppose that $\alpha\in\Z_{\geq0}^n$ is such that $f=bx^{\alpha}+g$, where $b\neq0$, and every exponent in $g$ is strictly greater than $\alpha$ with respect to $\preceq$. Let $\beta\in\Z_{\geq0}^n$. Then there is some polynomial $\widetilde{g}$ so that
\[ x^\beta f=bx^{\alpha+\beta}+\widetilde{g}\]
and every exponent in $\widetilde{g}$ is strictly greater than $\alpha+\beta$ with respect to $\preceq$.
\end{lemma}

\begin{proof}
This follows by observing that if $\alpha,\gamma\in\Z_{\geq0}^n$ satisfy $\alpha\prec\gamma$, then $\alpha+\beta\prec\gamma+\beta$.
\end{proof}

\section{A key lemma}

%\begin{lemma}
%Suppose that $P$ is a $(n-1)$-variable homogeneous polynomial of degree $r$ and that $P$ can be written in the form 
%\[ P(s)=\sum_{\j\in\Z_{\geq0}^{n-2}} \sum_{\alpha\in I_{\j}} P_\alpha(s) \]
%where for each $\alpha\in I_{\j}$, the polynomial $P_{\alpha}$ has smallest exponent $(\j,r-|\j|)$ with respect to $\preceq$. 
%\end{lemma}

\begin{lemma}
\label{warmuplemma2}
Let $k,n\in\N$. Let
$\mathcal{I}$ be a subset of $\Z_{\geq0}^n$ so that each $\alpha\in\mathcal{I}$ satisfies $|\alpha|<k(q-1)$, and for each $\alpha\in\mathcal{I}$, $|\alpha|$ is distinct. Let $c_\alpha\in\F_q$, and $b\in \F_q$ with $b\neq0$. Suppose for each $\beta\in\Z_{\geq0}^n$ with $|\beta|<k$ and each non-zero $\rho\in\F_q,$ the value
\[ f_\beta(\rho)=\sum_{\alpha\in\mathcal{I}}b^{\alpha_n-\beta_n}c_\alpha{\alpha\choose\beta}\rho^{|\alpha|-|\beta|}\]
is zero. 
Then $c_\alpha=0$ for all $\alpha\in\mathcal{I}$. 

\end{lemma}

\begin{proof}

Consider $f_0(\rho)=\sum_{\alpha\in\mathcal{I}}b^{\alpha_n}c_\alpha \rho^{|\alpha|}$, which we view as a univariate polynomial in $\rho$. As $|\alpha|<kq$ for each $\alpha\in\mathcal{I}$ by assumption, the degree of $f_0$ is less than $k(q-1)$. We will show that $f_0$ vanishes on $\F_q$ with multiplicity $k$, and then use the strengthened Schwartz-Zippel lemma to show $f_0$ is the zero polynomial.

To this end, let $0\leq w<k$.
Using Lemmas~\ref{expHasse} and~\ref{lemma:bincoeff}, the $w$-th Hasse derivative of $f_0$ is
\begin{align*} 
f_0^{(w)}(\rho)=\sum_{\alpha\in\mathcal{I}}b^{\alpha_n}c_\alpha{|\alpha|\choose w}\rho^{|\alpha|-w}&=\sum_{\alpha\in\mathcal{I}}b^{\alpha_n}c_\alpha \sum_{\beta:|\beta|=w}{\alpha\choose \beta} \rho^{|\alpha|-|\beta|}.
\end{align*}
Switching the sums and using the assumption that $f_\beta(\rho)=0$ for each $\rho\neq0$ and $|\beta|=w$, we see that
\[ f_0^{(w)}(\rho)=\sum_{\beta:|\beta|=w}\left(\sum_{\alpha\in\mathcal{I}}b^{\alpha_n}c_\alpha {\alpha\choose \beta} \rho^{|\alpha|-|\beta|}\right)=\sum_{\beta:|\beta|=w}b^{\beta_n}f_\beta(\rho)=0 \]
for each non-zero $\rho\in\F_q$.
Thus $f_0$ vanishes on $\F_q\setminus\{0\}$ with multiplicity $k$, as claimed. 
Therefore
\[ \sum_{\rho\in\F_q\setminus\{0\}}\text{mult}(f_0,\rho)\geq k(q-1) >\text{deg}(f_0)\]
and so by Lemma~\ref{SZ},
$f_0$ is the zero polynomial. 
By assumption, $|\alpha|$ is distinct for each $\alpha\in\mathcal{I}$, and so the coefficients of $f_0$  are $\{b^{\alpha_n}c_\alpha :\alpha\in\mathcal{I}\}$. Since each of these coefficients is zero, and $b\neq0$,  we have $c_\alpha=0$ for all $\alpha\in\mathcal{I}$, as desired.

\end{proof}

\section{Warm up argument in $\F_q^2$}

\label{secdim2}

\begin{proposition}
\label{warmup}
Let $q>2$ and let $S\subset\F_q^2$ be a set satisfying the following: for all $\rho\in\F_q$, there exists some $a\in \F_q^2$ such that
\begin{equation} 
\label{subseteq}
\{a+\rho(\lambda,\lambda^2):\lambda\in\F_q\}\subset S.
\end{equation}
Let $k\in\N$ be a multiple of $q$, let $D=k(q-1)-1$ and $M=3k-4k/q$. Then
\begin{equation} 
\label{warmupineq}
{M+1\choose 2}\cdot|S|\geq{D+2\choose 2}.
\end{equation}
\end{proposition}
We would obtain the bound from Theorem~\ref{thm} with $n=\ell=2$ by taking $k\to\infty$ in the inequality~(\ref{warmupineq}). We will complete this step in the full proof of Theorem~\ref{thm}, but skip it for this warm up argument.

%\begin{lemma}
%\label{warmuplemma}
%Let $S$, $k$, $M$, and $D$ be as in Proposition~\ref{warmup}. Suppose that $P$ is a non-zero polynomial of degree at most $D$ vanishing on $S$ with multiplicity $M$. Let $P^{(\beta)}$ denote the Hasse derivative of $P$ of order $\beta\in\Z_{\geq0}^n$, and let 
%\[ Q(t)=P^{(\beta)}(a+\rho(t,t^2)).\]
%Then for $|\beta|<k$, $Q$ is the zero polynomial.
%\end{lemma}

\begin{proof}[Proof of Proposition \ref{warmup}]
Towards contradiction, assume that the conclusion of the proposition is false. Then by Lemma~\ref{exists} there exists a non-zero polynomial $P\in \F_q[x_1,x_2]$ of degree at most $D$ vanishing on $S$ with multiplicity $M$. 
Observe that our choice of $M$ and $D$ satisfies the inequalities in~(\ref{ineqMDgen}). 
This follows after verifying the inequalities $\ell D<Mq$ and $\ell(D-k)<(M-k)q$, the first of which relies on the assumption that $q>2$.

Write
\[ P(x_1,x_2)=\sum_{|\alpha|\leq D} c_\alpha x_1^{\alpha_1}x_2^{\alpha_2}.\]
By Lemma~\ref{derivszero} with $\beta=(0,0)$, we see that
\begin{equation}
\label{Prho'}
P_\rho(t)=\sum_{|\alpha|\leq D}c_\alpha(a_1+\rho t)^{\alpha_1}(a_2+\rho t^2)^{\alpha_2}\end{equation}
is the zero polynomial.

Notice that the term in~(\ref{Prho'}) associated with $\alpha$ is a polynomial of degree at most $\alpha_1+2\alpha_2$. As such, we define
\[ \mathcal{A}_{j}=\{\alpha:|\alpha|\leq D, \; \alpha_1+2\alpha_2=j\}. \]
Since $P$ is non-zero, 
\[ m = \max\{j:\text{there exists }\alpha\in \mathcal{A}_j \text{ such that } c_\alpha\neq0\} \]
exists.

Now let $\beta\in\Z_{\geq0}^2$ be such that $|\beta|<k$ and let $P_\beta=P^{(\beta)}$ be the $\beta$-th Hasse derivative of $P$. Moreover, let $\rho\neq0$ and $P_{\beta,\rho}(t)=P_\beta(a+\rho(t,t^2))$, which we know is the zero polynomial by Lemma~\ref{derivszero}. 
Using the partition of the exponents $\alpha$ given by the sets $\mathcal{A}_j$, and recalling Lemma~\ref{expHasse}, we may write
\begin{equation}
\label{Prhored}
P_{\beta,\rho}(t)=\sum_{j=0}^m \sum_{\alpha\in\mathcal{A}_{j}}c_\alpha{\alpha\choose\beta}(a_1+\rho t)^{\alpha_1-\beta_1}(a_2+\rho t^2)^{\alpha_2-\beta_2}.
\end{equation}
Let $w=\beta_1+2\beta_2$. The summand in~(\ref{Prhored}) corresponding to $\alpha\in\mathcal{A}_j$ is a polynomial of degree at most $j-w$. Thus the highest power of $t$ present in~(\ref{Prhored}) is $m-w$, and it may only appear as the highest power term of a summand corresponding to $\alpha\in\mathcal{A}_m$. Therefore the coefficient of $t^{m-w}$ is
\[ g_{\beta}(\rho):=\sum_{\alpha\in\mathcal{A}_m}c_\alpha{\alpha\choose\beta}\rho^{|\alpha|-|\beta|}. \]
Since $P_{\beta,\rho}$ is the zero polynomial, $g_\beta(\rho)=0$ for all $|\beta|<k$ and all $\rho\neq0$. Moreover, for each $\alpha\in\mathcal{A}_m$, $|\alpha|$ will be distinct, and so Lemma~\ref{warmuplemma2} with $b=1$ implies that $c_\alpha=0$ for all $\alpha\in\mathcal{A}_m$, contradicting our choice of $m$. 

\end{proof}

\section{Proof of Theorem~\ref{thm} assuming Proposition~\ref{prop}}
\label{secthm}

Let $g\in\F_q[s_1,\dots,s_{n-1}]$ be a homogeneous degree $\ell$ polynomial, with $\ell\geq2$. Suppose $S\subset\F_q^n$ satisfies: for any $\rho\in\F_q$, there is some $a\in\F_q^n$ and some polynomial $g_\rho\in\F_q[s_1,\dots,s_{n-1}]$ with homogeneous part of highest degree equal to $g$, so that
\[ \{a+\rho (\lambda, g_\rho(\lambda)):\lambda\in\F_q^{n-1}\}\subset S. \]
Assume towards contradiction that there is some $k\in \N$, a multiple of $q$, so that for
\[ D=k(q-1)-1, \quad M=(\ell+1)k-2\ell k/q \]
we have
\begin{equation} 
\label{contassump}
{M+n-1\choose n}\cdot|S|<{D +n\choose n}.
\end{equation}
Then by Lemma~\ref{exists}, there exists a non-zero polynomial $P\in \F_q[x_1,\dots,x_n
]$ of degree at most $D$ vanishing on $S$ with multiplicity $M$. 
Observe that our choice of $M$ and $D$ satisfies the inequalities in~(\ref{ineqMDgen}). 
This follows after verifying the inequalities $\ell D<Mq$ and $\ell(D-k)<(M-k)q$, the first of which relies on the assumption that $q>\ell$.

Let $d$ be the degree of $P$, and suppose $P(x)=\sum_{|\alpha|\leq d} c_\alpha x^\alpha.$
For $\rho\in\F_q$, by Lemma~\ref{derivszero} with $\beta=(0,0)$, 
the $(n-1)$-variate polynomial 
\[ P_\rho(s)=P(a+\rho(s,g_\rho(s))) \]
is the zero polynomial.
We have
\begin{align}
\label{Prhog2}
P_\rho(s)&=\sum_{|\alpha|\leq d} c_\alpha (a'+\rho s)^{\alpha'}(a_n+\rho g_\rho(s))^{\alpha_n}
\end{align}
where $a'=(a_1,\dots,a_{n-1})$, and $\alpha'$ is defined similarly.
Since $g_\rho$ has degree $\ell$, the term in~(\ref{Prhog2})
associated with $\alpha$ is a polynomial in $s$ of degree at most $d(\alpha)=\alpha_1+\dots+\alpha_{n-1}+\ell \alpha_n$. 
We will partition the exponents $\alpha$ according to the value of $d(\alpha)$. As such, define
\[ \mathcal{A}_{j}=\{\alpha\in\Z_{\geq0}^n:\alpha_1+\dots+\alpha_{n-1}+\ell \alpha_n=j\}. \]
Since $P$ is non-zero, 
\[ m = \max\{j:\text{there exists }\alpha\in \mathcal{A}_j \text{ such that } c_\alpha\neq0\} \]
exists.

Now let $\beta\in\Z_{\geq0}^n$ be such that $|\beta|<k$, and let $P_\beta=P^{(\beta)}$ be the $\beta$-th Hasse derivative of $P$. Then Lemma~\ref{derivszero} implies that $P_{\beta,\rho}(s)=P_\beta(a+\rho(s,g_\rho(s)))$ is the zero polynomial. 
Using the partition of the exponents $\alpha$ given by the sets $\mathcal{A}_j$, and recalling Lemma~\ref{expHasse}, we may write
\begin{equation}
\label{Pbr}
P_{\beta,\rho}(s)=\sum_{j=0}^m \sum_{\alpha\in\mathcal{A}_{j}}c_\alpha{\alpha\choose\beta}(a'+\rho s)^{\alpha'-\beta'}(a_n+\rho g_\rho(s))^{\alpha_n-\beta_n},
\end{equation}
where $\beta'=(\beta_1,\dots,\beta_{n-1})$.

Now assume $\rho\neq0$, and let $w=|\beta'|+\ell\beta_n$. The summand in~(\ref{Pbr}) corresponding to $\alpha\in\mathcal{A}_j$ is a polynomial of degree at most $j-w$. Thus terms of degree $m-w$ can only appear as the highest degree term of summands in~(\ref{Pbr}) corresponding to $\alpha\in\mathcal{A}_m$. Therefore the part of $P_{\beta,\rho}$ of degree $m-w$ can be expressed as
\[ \widetilde{P}_{\beta,\rho}(s)=\sum_{\alpha\in\mathcal{A}_m}c_\alpha {\alpha\choose\beta} (\rho s)^{\alpha'-\beta'}(\rho g(s))^{\alpha_n-\beta_n}
\]
since $g$ is the homogeneous part of $g_\rho$ of highest degree (i.e., degree $\ell$). Since $P_{\beta,\rho}$ is the zero polynomial, the coefficient of each of its degree $m-w$ monomials is zero, so $\widetilde{P}_{\beta,\rho}$ is the zero polynomial as well.

Let $\widetilde{P}(x)=\sum_{\alpha\in\mathcal{A}_m}c_\alpha x^\alpha$. Notice that $\widetilde{P}_{\beta,\rho}(s)=\widetilde{P}^{(\beta)}(\rho(s,g(s)))$. 
Since $\widetilde{P}_{\beta,\rho}$ is the zero polynomial for all $|\beta|<k$ and all non-zero $\rho\in\F_q$, Proposition~\ref{prop} implies that $\widetilde{P}$ is the zero polynomial. But then
 $c_\alpha=0$ for all $\alpha\in\mathcal{A}_m$, contradicting the choice of $m$. Therefore our assumption in~(\ref{contassump}) is false, and so for every $k\in\N$ that is a multiple of $q$, we have
 \[ |S|\geq \dfrac{{D+n\choose n}}{{M+n-1\choose n}}=\dfrac{\prod_{i=1}^n(D+i)}{\prod_{i=1}^n(M-1+i)}=\dfrac{\prod_{i=1}^n(k(q-1)-1+i)}{\prod_{i=1}^n((\ell+1)k-2\ell k/q-1+i)}.\]
 Rewrite this inequality as
  \[ |S|\geq \dfrac{\prod_{i=1}^n(q-1+(i-1)/k)}{\prod_{i=1}^n(\ell+1-2\ell /q+(i-1)/k)}. \]
Since this inequality needs to hold for arbitrarily large $k$ that are multiples of $q$, we obtain $|S|\geq (q-1)^n/(\ell+1-2\ell/q)^n$.

\section{Proof of Proposition~\ref{prop}}
\label{secprop}

\subsection{Proof in dimension $n=2$}

Let $Q\in \F_q[x_1,x_2]$ be the polynomial of degree $d<k(q-1)$ and $f\in \F_q[t]$ be a non-zero homogeneous polynomial of degree $\ell$ as in the statement of Proposition~\ref{prop}, so that $Q_{\beta,\rho}(s)=Q^{(\beta)}(\rho(t,f(t)))$ is the zero polynomial for all non-zero $\rho\in\F_q$ and $|\beta|<k$. We may write $Q(x)=\sum_{\alpha\in I}c_\alpha x^\alpha$ where each $\alpha\in I$ satisfies
\begin{equation} 
\label{alphacond}
\alpha_1+\ell\alpha_2=m,
\end{equation}
and we may write $f(t)=bt^\ell$, with $b\neq0$. 

Let $\beta\in\Z_{\geq0}^2$ be such that $|\beta|<k$, and let $w=\beta_1+\ell \beta_2$. By Lemma~\ref{expHasse}, we have
\[ Q_{\beta,\rho}(t)=\sum_{\alpha\in I}c_\alpha{\alpha\choose\beta} (\rho t)^{\alpha_1-\beta_1}(\rho bt^\ell)^{\alpha_2-\beta_2}=\left(\sum_{\alpha\in I}b^{\alpha_2-\beta_2}c_\alpha {\alpha\choose\beta}\rho^{|\alpha|-|\beta|}\right )t^{m-w} \]
where the last equality uses (\ref{alphacond}). By assumption, this is the zero polynomial, and so we have 
\[ \sum_{\alpha\in I}b^{\alpha_2-\beta_2}c_\alpha {\alpha\choose\beta}\rho^{|\alpha|-|\beta|}=0 \quad \text{for each} \quad \rho\in\F_q\setminus\{0\}, \; |\beta|<k. \]
Moreover, since $Q$ has degree $d<k(q-1)$, each $\alpha\in I$ satisfies $|\alpha|<k(q-1)$, and as each $\alpha\in I$ satisfies (\ref{alphacond}) with $\ell\geq2$, we know that $|\alpha|$ is distinct for each $\alpha\in I$. Then by Lemma~\ref{warmuplemma2}, $c_\alpha=0$ for each $\alpha\in I$.

\subsection{Proof in dimensions $n>2$}

Let $x=(x_1,\dots,x_n)$ and $s=(s_1,\dots,s_{n-1})$. 
Let $Q\in\F_q[x]$ be a polynomial of degree $d<k(q-1)$ so that 
$Q(x)=\sum_{\alpha\in I }c_\alpha x^\alpha$ where each $\alpha\in I$ satisfies 
\begin{equation}
\label{exprule}
\alpha_1+\dots+\alpha_{n-1}+\ell\alpha_n=m.
\end{equation}
Let $f\in\F_q[s]$ be a homogeneous polynomial of degree $\ell$ so that for all $|\beta|<k$ and all $\rho\neq0$, 
\[ Q_{\beta,\rho}(s)=Q^{(\beta)}(\rho(s,f(s)))  \]
is the zero polynomial. 
By~(\ref{exprule}) and choice of $f$, we have that $Q_\rho(s):=Q(\rho(s,f(s)))$
is a homogeneous polynomial of degree $m$. As such, all its exponents take the form $(\i,m-|\i|)$ for $\i\in \Z_{\geq0}^{n-2}$. 
Henceforth, when we compare two exponent vectors, this comparison is with respect to the lexicographical order.

Let $e=(\widetilde{e},\ell-|\widetilde{e}|)\in\Z_{\geq0}^{n-2}\times \Z_{\geq0}$ be the smallest exponent occurring in $f$; that is, we may write 
\begin{equation} 
\label{decompf}
f(s) = bs^e+g(s) 
\end{equation}
where $b\neq0$, and either $g$ is the zero polynomial, or every exponent in $g$ is strictly greater than $e$. 
For $\alpha\in\Z_{\geq0}^n$, let $\alpha'=(\alpha_1,\dots,\alpha_{n-1})$, and similarly for $\beta$. 
By Lemma~\ref{expHasse} and then~(\ref{decompf}), we have
\begin{align}
Q_{\beta,\rho}(s)&=\sum_{\alpha\in I}c_\alpha{\alpha\choose\beta}(\rho s)^{\alpha'-\beta'}(\rho f(s))^{\alpha_n-\beta_n} \nonumber  \\ &=\sum_{\alpha\in I}c_\alpha{\alpha\choose\beta}\rho^{|\alpha|-|\beta|} s^{\alpha'-\beta'}(bs^e+g(s))^{\alpha_n-\beta_n}. \label{Qbr}
\end{align}
Let $\widetilde{Q}_{\alpha,\beta,\rho}(s)$ be the summand in~(\ref{Qbr}) corresponding to $\alpha\in I$. If there is some $i$ such that $\alpha_i<\beta_i$, then $\widetilde{Q}_{\alpha,\beta,\rho}(s)$ is identically zero. 
%%CHANGE Qtilde a,b to Qtilde a,b,p
Otherwise, by Lemma~\ref{lem-powerf}, we may write $f^{\alpha_n-\beta_n}$ as 
\[ (f(s))^{\alpha_n-\beta_n} = b^{\alpha_n-\beta_n} s^{(\alpha_n-\beta_n)e} +g_{\alpha_n-\beta_n}(s) \]
where either $g_{\alpha_n-\beta_n}$ is the zero polynomial, or every exponent in $g_{\alpha_n-\beta_n}$ is strictly greater than $(\alpha_n-\beta_n) e$. By Lemma \ref{lem-xbf}, if $c_\alpha\neq0$, the smallest exponent appearing in
\[ \widetilde{Q}_{\alpha,\beta,\rho}(s)=c_\alpha{\alpha\choose\beta}\rho^{|\alpha|-|\beta|}s^{\alpha'-\beta'}\left(b^{\alpha_n-\beta_n}s^{(\alpha_n-\beta_n)e}+g_{\alpha_n-\beta_n}(s)\right)\]
is 
\[ d(\alpha,\beta):=\alpha'-\beta'+(\alpha_n-\beta_n)e \]
and its corresponding coefficient is 
\begin{equation}
\label{mincoeff}
 b^{\alpha_n-\beta_n}c_\alpha{\alpha\choose\beta}\rho^{|\alpha|-\beta|}.    
\end{equation}

We will partition $I$ using values of $d(\alpha,\vec{0})$ into sets $I_{\j}$. 
For $\j\in\Z_{\geq0}^{n-2}$, let $I_{\j}$ be the set of $\alpha\in I$ satisfying
\[ d(\alpha,\vec{0}) = (\j,m-|\j|);\]
in particular, each $\alpha\in I_{\j}$ has the form
\begin{equation}
\label{formalpha}
\alpha=(j_1-\alpha_n e_1, \dots,j_{n-2} -\alpha_n e_{n-2},m-|\j|-\alpha_n e_{n-1},\alpha_n).
\end{equation}
Notice that in this equation, $\j$, $m$, and $e$ are fixed, while $\alpha_n$ may take on any nonnegative integer value that is at most $d$, the degree of $Q$, and that makes each component of~(\ref{formalpha}) nonnegative. Using the relationship 
\[ d(\alpha,\beta)=d(\alpha,\vec{0})-\beta'-\beta_ne\]
and letting $\widetilde{\beta}=(\beta_1,\dots,\beta_{n-2})$, we see that
if $\alpha\in I_{\j}$, then either $\widetilde{Q}_{\alpha,\beta,\rho}$ is the zero polynomial, or has smallest exponent
%% CLARIFY THIS - CITE d(alpha,beta) and relationship with d(alpha,0)
\begin{equation}
\label{minexp}
d_{\j}(\beta):=(\j-\widetilde{\beta}-\beta_n\widetilde{e}, m-|\j|-\beta_{n-1}-\beta_ne_{n-1}).
\end{equation}
Moreover, if $\j\prec \i$, then $\j-\widetilde{\beta}-\beta_n\widetilde{e}\prec \i-\widetilde{\beta}-\beta_n\widetilde{e}$, and so 
\begin{equation}
\label{compdjb}
\text{if }\; \j\prec\i, \; \text{ then } \; d_{\j}(\beta)\prec d_{\i}(\beta).
\end{equation}

Now assume towards contradiction that there is some $\alpha\in I$ so that $c_\alpha\neq0$. By Proposition~\ref{lexwellord}, the minimum (with respect to $\preceq$) of the set
\[ \{\i: \exists \alpha\in I_{\i} \; \text{ s.t.}\; c_\alpha\neq0 \} \]
exists; we will call this minimum exponent $\j$. By choice of $\j$, we can write
\[ Q_{\beta,\rho}(s)=\sum_{\i\succeq\j}\sum_{\alpha\in I_{\i}}\widetilde{Q}_{\alpha,\beta,\rho}(s). \]
To derive a contradiction, we will consider the coefficient of $s^{d_{\j}(\beta)}$ in $Q_{\beta,\rho}$.

Now let $|\beta|<k$, and first assume that
\begin{equation} 
\label{betaineq}
\beta_i+\beta_ne_i\leq j_i  \text{ for all }  1\leq i \leq n-2, \text{ and } \beta_{n-1}+\beta_ne_{n-1}\leq m-|\j|. 
\end{equation}
These inequalities guarantee each component of $d_{\j}(\beta)$ is non-negative. If $\i\succ\j$ and $\alpha\in I_{\i}$, then the exponent $d_{\j}(\beta)$ does not appear in $\widetilde{Q}_{\alpha,\beta,\rho}$, by~(\ref{compdjb}). Therefore we need only consider $\alpha\in I_{\j}$. For such $\alpha$, the coefficient of $s^{d_{\j}(\beta)}$ is given in~(\ref{mincoeff}). Thus the coefficient of $s^{d_{\j}(\beta)}$ in $Q_{\beta,\rho}$ is 
\[ q_\beta(\rho):=\sum_{\alpha\in I_{\j}}b^{\alpha_n-\beta_n}c_\alpha{\alpha\choose\beta} \rho^{|\alpha|-|\beta|},\]
and as we assume $Q_{\beta,\rho}$ is the zero polynomial for all non-zero $\rho$, $q_\beta(\rho)=0$ for all non-zero $\rho$.

Now suppose one of the inequalities in~(\ref{betaineq}) fails to hold, say $\beta_i+\beta_ne_i>j_i$. Then using the form of $\alpha$ in~(\ref{formalpha}), we'd have
\[ \beta_i+\beta_ne_i>\alpha_i+\alpha_ne_i,\]
in which case one of $\beta_i>\alpha_i$ or $\beta_n>\alpha_n$ must hold. Then ${\alpha\choose\beta}=0$, and so we'd still have $q_\beta(\rho)=0$ for all non-zero $\rho$.

Therefore $q_{\beta}(\rho)=0$ for all $\rho\in \F_q\setminus\{0\}$ and all $|\beta|<k$. Moreover, since $Q$ has degree $d<k(q-1)$, each $\alpha\in I_{\j}$  satisfies $|\alpha|<k(q-1)$. 
Further, the form of $\alpha\in I_{\j}$ in (\ref{formalpha}) implies that $\alpha_n$ is distinct for each $\alpha\in I_{\j}$, and so using~(\ref{exprule}) and recalling that $\ell\geq2$, we see that
\[ |\alpha|=m-(\ell-1)\alpha_{n} \]
is distinct for each $\alpha\in I_{\j}$. 
Then Lemma~\ref{warmuplemma2} implies that $c_\alpha=0$ for all $\alpha\in I_{\j}$, contradicting our choice of $\j$. 

\section{Acknowledgements}
This work was funded by NSERC Discovery Grants 22R80520 and GR010263. Many thanks to my advisor, Malabika Pramanik, for very helpful discussions while working on this problem and preparing this article.
%% PUT LEMMA 3.2 in own section; call it "a key proposition" 

%%% Make sure Kakeya is capitalized

\bibliographystyle{amsplain}

\bigskip

\noindent{\sc Department of Mathematics, UBC, Vancouver,
B.C. V6T 1Z2, Canada}

\smallskip

\noindent{\it  ctrainor@math.ubc.ca}

\end{document}